\documentclass[12pt,reqno]{amsart}

\usepackage{amssymb,amsmath,amsthm}
\usepackage{amsfonts}
\usepackage{latexsym}
\usepackage{epsfig}
\usepackage{amsfonts}
\usepackage{subfig}

\usepackage[left=3cm,top=3cm,right=3cm,bottom=3cm]{geometry}

\newtheorem{theorem}{Theorem}

\newtheorem{lemma}[theorem]{Lemma}
\newtheorem{corollary}[theorem]{Corollary}

\newcommand{\ilm}[3]{{\rm ILM}_{#1,#2}(#3)}
\newcommand{\ilt}[2]{{\rm ILT}_{#1}(#2)}
\newcommand{\ilat}[2]{{\rm ILAT}_{#1}(#2)}

\newcommand{\vol}{{\rm vol}}

\begin{document}
	
\title{The Iterated Local Model for Social Networks}
\author[Anthony Bonato]{Anthony Bonato}
\address{Department of Mathematics\\
Ryerson University\\
Toronto, ON\\
Canada, M5B 2K3} \email{abonato@ryerson.ca}
\author[Huda Chuangpishit]{Huda Chuangpishit}
\address{Department of Mathematics\\
Ryerson University\\
Toronto, ON\\
Canada, M5B 2K3} \email{h.chuang@ryerson.ca}
\author[Sean English]{Sean English}
\address{Department of Mathematics\\
Ryerson University\\
Toronto, ON\\
Canada, M5B 2K3} \email{sean.english@ryerson.ca}
\author[Bill Kay]{\,\hspace{3cm}\,Bill Kay}
\address{Department of Mathematics\\
Ryerson University\\
Toronto, ON\\
Canada, M5B 2K3} \email{bill.kay@ryerson.ca}
\author[Erin Meger]{Erin Meger}
\address{Department of Mathematics\\
Ryerson University\\
Toronto, ON\\
Canada, M5B 2K3} \email{erin.k.meger@ryerson.ca}
\thanks{The authors gratefully acknowledge support from NSERC, Ryerson University, and the Fields Institute for Research in Mathematical Sciences}

\begin{abstract}
On-line social networks, such as in Facebook and Twitter, are often studied from the perspective of friendship ties between agents in the network. Adversarial ties, however, also play an important role in the structure and function of social networks, but are often hidden. Underlying generative mechanisms of social networks are predicted by structural balance theory, which postulates that triads of agents, prefer to be transitive, where friends of friends are more likely friends, or anti-transitive, where adversaries of adversaries become friends. The previously proposed Iterated Local Transitivity (ILT) and Iterated Local Anti-Transitivity (ILAT) models incorporated transitivity and anti-transitivity, respectively, as evolutionary mechanisms. These models resulted in graphs with many observable properties of social networks, such as low diameter, high clustering, and densification.

We propose a new, generative model, referred to as the Iterated Local Model (ILM) for social networks synthesizing both transitive and anti-transitive triads over time. In ILM, we are given a countably infinite binary sequence as input, and that sequence determines whether we apply a transitive or an anti-transitive step. The resulting model exhibits many properties of complex networks observed in the ILT and ILAT models. In particular, for any input binary sequence, we show that asymptotically the model generates finite graphs that densify, have clustering coefficient bounded away from 0, have diameter at most 3, and exhibit bad spectral expansion. We also give a thorough analysis of the chromatic number, domination number, Hamiltonicity, and isomorphism types of induced subgraphs of ILM graphs.
\end{abstract}
	
\maketitle

\section{Introduction}

Friendship ties are essential constructs in on-line social networks, as witnessed by friendship between Facebook users or followers on Twitter. Another pervasive aspect of social networks are negative ties, where users may be viewed as adversaries, competitors, or enemies. For further background on ties in social networks and more generally, complex networks, see \cite{bbook,at,ek}. Negative ties are often hidden, but may have a powerful influence on the social network. An early example of how negative ties influences networks structure comes from the famous Zachary Karate network, where an adversarial relationship assisted in the formation of two distinct communities \cite{zachary}.

Complex networks, including on-line social networks, contain numerous mechanisms governing edge formation. In the literature, models for complex networks have exploited principles of preferential attachment \cite{ba,bol}, copying or duplication \cite{ilt,ilat,CL}, or geometric settings \cite{SPA,MGEOP,geop,k,mmp,Zhang}. The majority of complex network models are premised on the formation of edges via positive ties. Structural balance theory in social network analysis cites several mechanisms to complete \emph{triads}, or triples of vertices. Vertices may have positive or negative ties, and triads are balanced if the signed product of their ties is positive. Hence, balanced triads are those consisting of all friends, or two enemies and a friend. These triads reflect the folkloric adages ``friends of friends are friends'' and ``enemies of enemies are more likely friends,'' respectively. Such triad closure is suggestive of an analysis of adversarial relationships between vertices as another model for edge formation. For an example, consider market graphs, the vertices are stocks, and stocks are
adjacent as a function of their correlation measured by a threshold value $\theta \in (0,1).$ Market graphs were considered in the case of negatively correlated (or competing) stocks, where stocks
are adjacent if $\theta < \alpha,$ for some positive $\alpha$; see \cite{market}. In social networks, negative correlation corresponds to enmity or rivalry between agents. We may also consider opposing networks formed by nation states or rival organizations, or alliances formed by mutually shared adversaries as in the game show Survivor~\cite{surv,guo}.

Transitivity is a pervasive and folkloric notion in social networks, and postulates triads with all positive signs. A simplified, deterministic model for transitivity
was posed in \cite{ilt1,ilt}, where vertices are added over time, and for each vertex $x$, there is a \emph{clone} that is adjacent to $x$ and all of its neighbors. The resulting Iterated Local Transitivity (or ILT) model, while elementary to define, simulates many properties of social and other complex networks. For example, as shown in \cite{ilt}, graphs generated by the model densify over time, and exhibit bad spectral expansion. In addition, the ILT model generates graphs with the small world property, which requires the graphs to have low diameter and dense neighbor sets. For further properties of the ILT model, see \cite{BJR,mason}.

Adversarial relationships may be modeled by non-adjacency, and so we have the resulting closure of the triad as described in Figure~\ref{at}.
\begin{figure}[h!]
\begin{centering}
    \includegraphics{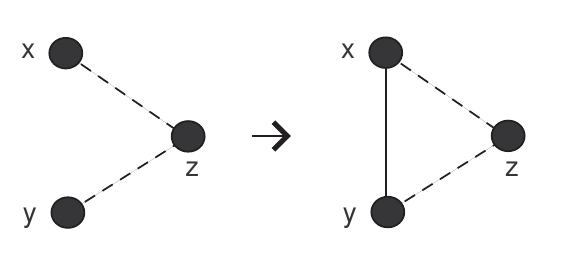}
    \caption{Vertices $x$ and $y$ share $z$ as a mutual adversary (denoted by dotted lines), and so form an alliance (denoted by an edge).}
    \label{at}
    \end{centering}
\end{figure}
A simplified, deterministic model simulating anti-transitivity in complex networks was introduced in \cite{ilat}. The Iterated Local Anti-Transitivity (or ILAT) model duplicates vertices $x$ in each time-step
by forming \emph{anti-clone} vertices. The anti-clone of $x$ is adjacent to the non-neighbor set of $x$. Perhaps unexpectedly, graphs generated by ILAT model exhibit many properties of complex networks such as densification, small world properties, and bad spectral expansion.

In the present paper, we consider a new model synthesizing both the ILT and ILAT models, allowing for both transitive and anti-transitive steps over time. We refer to this model as the Iterated Local Model (ILM), and define it precisely in the next subsection. Informally, in ILM we are given as input an infinite binary sequence. For each positive entry in the sequence, we take a transitive, ILT-type step. Otherwise, we take an anti-transitive, ILAT-type step. Hence, ILM contains both the ILT and ILAT model as special cases, but includes infinitely many (in fact, uncountably many) other model variants as a function of the infinite binary sequence.

We consider only finite, simple, undirected graphs throughout the paper. For a graph $G$ with vertex $v$, define the \emph{neighbor set} of $v$, written $N_G(v)$, to be $\{ u \in V(G): uv \in E(G)\}.$  The
\emph{closed neighbor set} of $v$, written $N_G[v],$ is the set $N_G(v) \cup \{v\}.$ Given a graph $G$, we denote its complement by $\overline{G}.$ When it is clear from context, we suppress the subscript $G$. For background on graph theory, the reader is directed to~\cite{west}. Additional background on complex networks may be found in the book
\cite{bbook}.
	
\subsection{The Iterated Local Model}

We now precisely define the iterated local model (ILM). First, we must define two iterative procedures on a graph $G$ by considering steps that are locally transitive or locally anti-transitive.

We define a graph $\mathrm{LT}(G)$ as follows. For each $x \in V(G)$, add a new vertex $x^\prime$ to the vertex set of $\mathrm{LT}(G)$ such that $x'$ is adjacent to all neighbors of $x$ in $G.$ In particular, $N_{\mathrm{LT}(G)}(x^\prime) = N_G[x]$. The vertex $x^\prime$ is called the {\em transitive clone} of $x$, and the resulting graph $\mathrm{LT}(G)$ is called the graph obtained from $G$ by applying one {\em locally transitive step}.

Analogously, we define the graph $\mathrm{LAT}(G)$ as follows. For each $x \in V(G)$, add a new vertex $x^\ast$ that is adjacent to all non-neighbors of $x$ in $G.$ In particular, $N_{\mathrm{LAT}(G)}(x^\ast) = V(G)\backslash N_G[x]$. The vertex $x^\ast$ is called the {\em anti-transitive clone} of $x$, and the resulting graph $\mathrm{LAT}(G)$ is called the graph obtained from $G$ by applying one {\em locally anti-transitive step}.

Note that the ILT model is defined precisely by applying iterative locally transitive steps; an analogous statement holds for the ILAT model. Observe also that clones or anti-clones introduced in the same time-step are pairwise non-adjacent.

\begin{figure}[h!]
\begin{center}

\subfloat{{\includegraphics{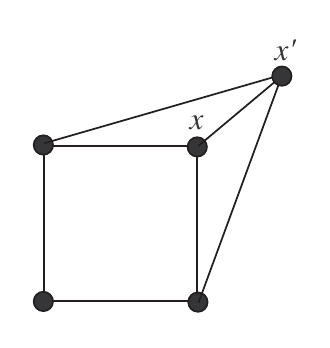} }}%
    \qquad
    \subfloat{{\includegraphics{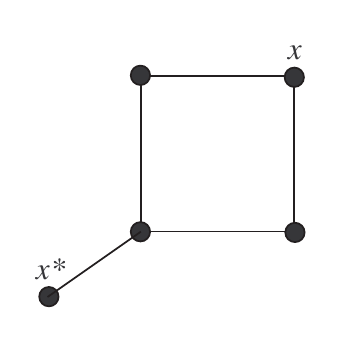} }}%

\caption{The vertex $x'$ is the clone of $x$ and the vertex $x^*$ is the anti-clone of $x$.}\label{fclone}
\end{center}
\end{figure}

Fix an infinite binary sequence $S=(s_0, s_1, s_2, \ldots)$, which we refer to as an \emph{input sequence}. The $t\textsuperscript{th}$ \emph{Iterated Local Model}(ILM) graph, denoted $\ilm{t}{S}{G}$ is defined recursively. In the case $t=0,$ we have that $\ilm0S{G}=G$. For $t\geq 0$, we have that
		\[
		\ilm{t+1}{S}{G} = \begin{cases} \mathrm{LAT}(\ilm{t}{S}{G}) & \text{ if } s_{t} = 0,\\
		\mathrm{LT}(\ilm{t}{S}{G}) & \text{ otherwise.}
		\end{cases}
		\]
Hence, an instance of 1 in the sequence results in a transitive step; otherwise, an anti-transitive step is taken. If the sequence $S$ contains only 1's, then the resulting graph at each time-step is isomorphic to the graph from the same time-step of the ILT model, thus, we write $\ilt{t}G=\ilm{t}SG$. Similarly, if the sequence $S$ contains only 0's, the resulting graph is isomorphic to the graph from the same time-step of the ILAT model and we write $\ilat{t}G=\ilm{t}SG$. We use the simplified notation \emph{ILM graph} for a graph $\ilm{t}{S}{G}$ for any choice of $t$, $S$, or $G$. \emph{ILT} and \emph{ILAT graphs} are defined in an analogous fashion.

\begin{figure}[h!]
\begin{center}
\epsfig{figure=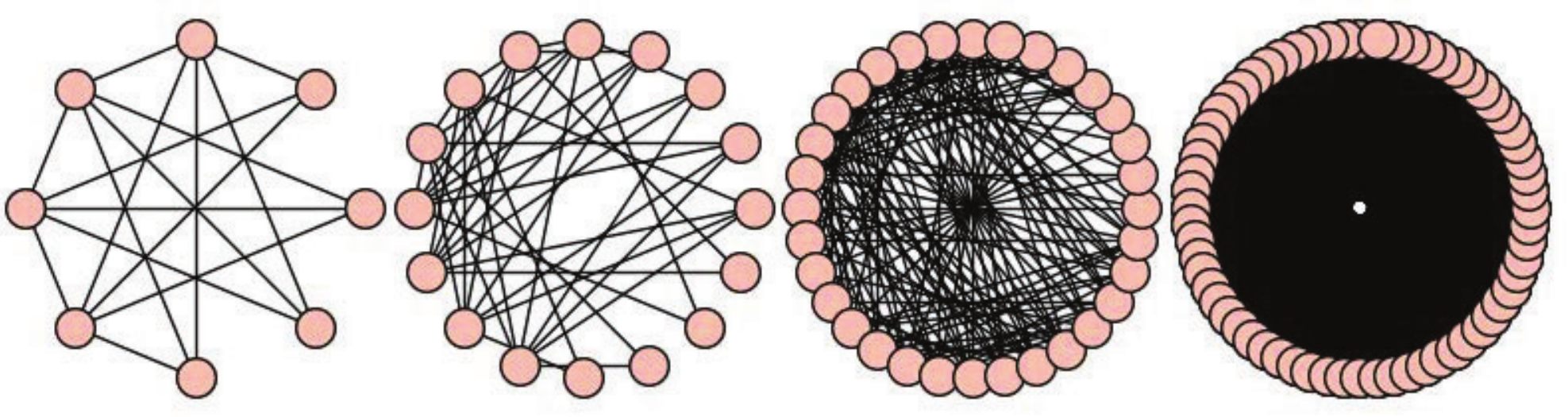,scale=0.75} \caption{ILM graphs at time-steps 1, 2, 3, and 4, with input sequence $S=(0,1,0,1,\ldots)$ and initial graph $C_4$ (not shown).}\label{fclone}
\end{center}
\end{figure}

We will denote the order of the $t\textsuperscript{th}$ ILM graph as $n_{t,S,G}=|V(\ilm{t}SG)|=2^t|V(G)|,$ and denote the size (that is, number of edges) by $e_{t,S,G}=|E(\ilm{t}SG)|.$ The (open) neighborhood of a vertex at time $t$ by
\[
N_{t,S,G}(v)=\{u\in V(\ilm{t}SG): uv\in E(\ilm{t}SG)\},
\]
and the closed neighborhood by $N_{t,S,G}[v]=N_{t,S,G}(v)\cup\{v\}.$ The degree of a vertex $v \in V(\ilm{t}SG)$ is written as $\deg_{t,S,G}(v)=|N_{t,S,G}(v)|.$  To simplify this notation, when the sequence $S$ and initial graph $G$ are clear from context, we will simply write $n_t$, $e_t$ $N_t(v)$, $N_t[v]$ and $\deg_t(v)$ for the order, the size, the neighborhoods, and the degree of a vertex at time $t$, respectively. We emphasize that $\ilm{t-1}SG$ is an induced subgraph of $\ilm{t}SG$, and as such we will always consider $\ilm{t-1}SG$ to be embedded in $\ilm{t}SG$ in the natural way. That is, expressions of the form $\deg_{t-1}(v)$ may be used for vertices $v\in V(\ilm{t}SG)$ to refer to the degree in the induced subgraph isomorphic to $\ilm{t-1}SG$.

In the present paper, we analyze various properties of ILM graphs. While the ILT and ILAT graphs satisfy various properties such as low diameter and densification, those statements are not a priori obvious for the ILM model.  We prove in Theorem~\ref{theorem density} that ILM generates graphs that densify over time. We do this by deriving the asymptotic size of ILM graphs given by the following expression: \[ |E({\ilm{t}{S}{G}})|=\Theta\left(2^{t+\beta}\left(\frac{3}2\right)^{t-\beta}\right),\]
where $\beta$ is the largest index less than $t$ such that $s_{\beta}=0$. The clustering coefficient of ILM graphs is studied in Section~\ref{Csec}. We say $S$ has \emph{bounded gaps between $0$'s}, or simply \emph{bounded gaps} if there exists some constant $k=k(S)$ such that there is no string of $k$ contiguous $1$'s. In contrast to the known clustering results for the ILT model in \cite{ilt}, we will see that for any sequence with bounded gaps, the clustering coefficient is bounded away from $0$. Our results here are also the first rigorously presented results for the clustering coefficient of the ILAT model \cite{ilat}; further, our results give an improvement on bounds for the clustering coefficient of ILT graphs.

Graph theoretical properties of ILM graphs are of interest in their own right. In Section~\ref{GPsec}, we explore classical graph parameters, including the chromatic and domination numbers, and diameter. The domination number of ILM graphs is eventually either 2 or 3, and Theorem~\ref{theorem if domination=2} classifies exactly when each value occurs. The diameter of an ILM graph eventually becomes $3$, usually after only two anti-transitive steps; see Theorem~\ref{theorem diameter after two time-steps}. Bad spectral expansion for the ILM graphs is proven in Section~\ref{SEsec}. Section~\ref{SRsec} includes a discussion of the Hamiltonicity of ILM graphs, proving that eventually, all ILM graphs are Hamiltonian. In addition, ILM graphs eventually contain isomorphic copies of any fixed finite graph. Hamiltonicity and induced subgraph properties were not perviously investigated in the ILT or ILAT models. We finish with a section of open problems and further directions.

\section{Density and Densification}\label{Dsec}

Complex networks often exhibit densification, where the number of edges grows faster than the number of vertices; see \cite{les1}. Both the ILT \cite{ilt} and ILAT models \cite{ilat} generate sequences of graphs whose edges grow super-linearly in the number of vertices. We now show a number of results regarding the number of edges in an ILM graph. The following theorem from \cite{ilt} gives us the average degree of an ILT graph, which will be useful for studying ILM graphs. We define the \emph{volume} of a graph $G$ by
\[
\mathrm{Vol}(G)=\sum_{x\in V(G)}\deg(x) = 2|E(G)|.
\]
Note that the average degree of $G$ equals $\mathrm{Vol}(G)/|V(G)|.$

\begin{theorem}\cite{ilt}\label{theorem ILT average degree}
	Let $G$ be a graph. For all integers $t\geq 0$, the average degree of $\ilt{t}G$ equals
	\[
	\left(\frac{3}2\right)^t\left(\frac{\mathrm{Vol}(G)}{n_0}+2\right)-2.
	\]
\end{theorem}

The analogous theorem for ILAT graphs is the following.

\begin{theorem}\cite{ilat}\label{theorem ILAT average degree}
	Let $G$ be a graph. For all integers $t\geq 0$, the average degree of $\ilat{t}G$ equals
	\[
	2^{t}\left( \frac{2n_0}{5} - o(1) \right).
	\]
\end{theorem}

We now provide an asymptotic formula for the number of edges in an ILM graph.
\begin{theorem}\label{theorem density}
	Let $S = (s_0, s_1, s_2,\dots)$ be a binary sequence with at least one $0$, let $\tau$ be the least index such that $s_{\tau}=0$, and let $t \geq \tau$. Let $\beta=\beta(t)$ be the largest index less than or equal to $t$ such that $s_{\beta} = 0$. For any graph $G$ and $t\geq \tau$,
	\[
	|E({\ilm{t}{S}{G}})|=\Theta\left(2^{t+\beta}\left(\frac{3}2\right)^{t-\beta}\right)=\Theta\left(2^{\beta}\left(\frac{3}2\right)^{t-\beta}n_t\right) .
	\]
\end{theorem}

\begin{proof}
Let $G_t=\ilm{t}SG$. Suppose in the first case that $s_t=0$ (that is, when $t=\beta$). We then have that for all $v\in V(G_{t-1})$, $\deg_{t}(v)=n_{t-1}-1$, since for all vertices $u,v\in V(G_{t-1})$, $u$ is a non-neighbor of $v$ if and only if the clone of $u$ is adjacent to $v$. Therefore, we have that
	
	\begin{align*}
	2|E(G_t)|&\geq \sum_{v\in V(G_{t-1})}\deg_{t}(v)\\
	&=\frac{n_t}2(n_{t-1}-1)\\
	&=\frac{n_t(n_t-2)}4.
	\end{align*}
We then have that
	\begin{align}
	\label{equation density after anti-transitive step}
	|E(G_t)|=\Theta(n_t^2)=\Theta(2^{2t})=\Theta(2^{t+\beta}),
	\end{align}
and this case is finished.

Suppose in the next and final case that $t> \beta$ so that $s_t=1$. By \eqref{equation density after anti-transitive step}, we have $|E(G_{\beta})|\geq c 2^{2\beta}$ for some constant $c>0$. By Theorem~\ref{theorem ILT average degree}, applied with $G=G_{\beta}$ and $t=t-\beta$, we derive that the average degree of $G_t$ is at least
	\[
	\left(\frac{3}2\right)^{t-\beta}\left(\frac{2c2^{2\beta}}{2^{\beta}|V(G)|}+2\right)-2\geq \frac{2c}{|V(G)|}2^{\beta}\left(\frac{3}2\right)^{t-\beta}.
	\]
	Thus, we have that
	\begin{align*}
	2 |E(G_t)|&\geq 2^t\ \left(\frac{2c}{|V(G)|}2^{\beta}\left(\frac{3}2\right)^{t-\beta}\right)\\
	&=\Omega\left(2^{t+\beta}\left(\frac{3}2\right)^{t-\beta}\right).
	\end{align*}
	
Similarly, by \eqref{equation density after anti-transitive step}, we have that there exists a constant $C$ such that $|E(G_{\beta})|\leq C2^{2\beta}$, so by Theorem~\ref{theorem ILT average degree}, the average degree is at most
	\[
	\left(\frac{3}2\right)^{t-\beta}\left(\frac{2C2^{2\beta}}{2^{\beta}|V(G)|}+2\right)-2\leq \left(\frac{3}2\right)^{t-\beta}\left(\frac{2C2^{\beta}}{|V(G)|}+2\right),
	\]
	so
	\begin{align*}
	2 |E(G_t)|&\leq 2^t \left(\left(\frac{3}2\right)^{t-\beta}\left(\frac{2C 2^{\beta}}{|V(G)|}+2\right)\right)\\
	&=O\left(2^{t+\beta}\left(\frac{3}2\right)^{t-\beta}\right),
	\end{align*}
completing the proof of Theorem~\ref{theorem density}.
\end{proof}

As a corollary to Theorem~\ref{theorem density}, we have that ILM graphs also undergo densification for any input sequence $S$.

\begin{corollary}\label{corollary densification}
	For any binary sequence $S$ and initial graph $G$, we have that
	\[
	\lim_{t\to\infty}\frac{|E(\ilm{t}S{G})|}{|V(\ilm{t}S{G})|}=\infty.
	\]
\end{corollary}

We may also consider sequences with bounded gaps between $0$'s. Results in this case follows immediately from the fact that for such sequences, $\beta=t-O(1)$.

\begin{corollary}
If $S$ is a binary sequence with bounded gaps between 0's, then for any graph $G$, $|E(\ilm{t}S{G})|= \Theta (n_t^2)$, where the implied constant depends on the size of the largest gap in $S$.
\end{corollary}

As noted in \cite{ilt} and \cite{ilat}, the size of ILT and ILAT models satisfy certain recurrences. More precisely, if $G$ is a graph on $n$ vertices with $e$ edges, then
\begin{equation}\label{equation transitive edge recurrence}
|E(\mathrm{LT}(G))| = 3 e - n,
\end{equation}
and
\begin{equation}\label{equation anti edge recurrence}
|E(\mathrm{LAT}(G))| = n^2 - e - n.
\end{equation}

When given a specific input sequence, we can use these recurrences to say much more about the number of edges. Here, we give a much stronger asymptotic result for alternating sequences.

\begin{theorem}\label{even}
	If we consider the alternating sequence $S=(1,0,1,0,1, \dots)$, then for even $t\ge 0$ we have that
	\[
	|E(\ilm{t}S{K_1})| = (1+o(1))\frac{16}{19} 2^{2t-2}=(1+o(1))\frac{16}{38}\binom{n_t}2.
	\]
\end{theorem}

\begin{proof}
Since $t$ is even, we have that $s_t=1$.  By \eqref{equation transitive edge recurrence} and \eqref{equation anti edge recurrence} we have that
	
	\begin{align*}
		e_{t+2} & = {n_{t+1}}^2-e_{t+1}-n_{t+1} \\
		& =  2^{2t+2} - (3e_t + 2^t) - 2^{t+1} \\
		& =  2^{2t+2} - 3e_t - 3 \cdot 2^{t} \\
		& =  \sum_{i = 0}^{t/2} 2^{2t -4i + 2} - 3 \left(\sum_{i = 0}^{t/2} 3^i 2^{t - 2i} \right).
	\end{align*}
	
The first geometric series converges to the dominant term $(16/19) \cdot 2^{2t+2}$ and the result follows.
\end{proof}

\section{Clustering Coefficient}\label{Csec}

Given a graph $G$, the \emph{local clustering coefficient of a vertex} $x\in V(G)$ is defined by
\[
c_G(x) = \frac{\left|E \left(G[N(x)] \right)\right|}{\binom{\deg(x)}2}.
\]
That is, $c_G(x)$ gives a normalized count of the number of edges in the subgraph induced by the neighbor set of $x$ in $G$. The \emph{clustering coefficient} of $G$ is given by
\[
C(G) = \frac{1}{|V(G)|} \sum_{x \in V(G)} c_G(x).
\]

Complex networks often exhibit high clustering, as measured by their clustering coefficients \cite{bbook}. Informally, clustering measures local density. For the ILT model, the clustering coefficient tends to $0$ as $t\to \infty$, although it does so at a slower rate than binomial random graphs with the same average degree \cite{ilt}. More precisely, we have the following.

\begin{theorem}\cite{ilt} \label{theorem ILT clustering}
\[
	\Omega\left(\left(\frac{7}8\right)^tt^{-2}\right)=C(\ilt{t}G)=O\left(\left(\frac{7}8\right)^tt^{2}\right)
\]
\end{theorem}

In contrast to Theorem~\ref{theorem ILT clustering}, we will see that for any sequence with bounded gaps, the clustering coefficient is bounded away from $0$ for the ILM graphs. We will find it useful to write $c_{t,S,G}(x)$ for $c_{\ilm{t}SG}(x)$, and when $S$ and $G$ are clear from context, we may simply write $c_t(x)$, consistent with earlier defined notation. We establish a bound on the change in the clustering coefficient from performing a transitive step.

\begin{lemma}\label{lemma transitive clustering change}
	If $G$ is a graph with minimum degree $\delta$, then
	\[
	C(\mathrm{LT}(G))\geq \left(\frac{7}8-\frac{3}{8\delta}\right)C(G).
	\]
\end{lemma}

\begin{proof}
	Let $x\in V(G)$, and let $x'$ be its transitive clone in $\mathrm{LT}(G)$. Let $c_0(x)$ denote the local clustering coefficient of $x$ in $G$, while $c_1(x)$ will denote the local clustering coefficient of $x$ in $\mathrm{LT}(G)$.

Note that $c_1(x')\geq c_0(x)$ since $N_1(x')=N_0[x]$, and $x$ is a dominating vertex in $G[N_0[x]]$. Recall that $N_t(x)$ and $N_t[x]$ are the open and closed neighborhoods of $x$ in $\ilm{t}SG$, in this case since we are only doing a single transitive step, we will only use these expressions for $t=0$ or $t=1$. We will give a lower bound for $c_1(x)$ in terms of $c_0(x)$.

We now count the edges in $N_1(x)$. In the induced subgraph $\mathrm{LT}(G)[N_0(x)]$, there are
\[
c_0(x)\binom{\deg_0(x)}{2}
\]
many edges. There are
\[
2c_0(x)\binom{\deg_0(x)}{2}+\deg_0(x)
\]
many edges between vertices in $N_0(x)$, and $\left(N_1(x)\setminus N_0(x)\right)\setminus \{x'\}$. Finally, there are $\deg_0(x)$ edges incident with $x'$. This accounts for all the edges in $N_1(x)$. Since $\deg_1(x)=2\deg_0(x)+1$, this gives us that for all $x\in V(G)$,
\begin{align*}
c_1(x)&=\cfrac{3\displaystyle \binom{\deg_0(x)}{2}c_0(x)+2\deg_0(x)}{\displaystyle\binom{2\deg_0(x)+1}2}\\
&=\cfrac{3(\deg_0(x)-1)c_0(x)+4}{2(2\deg_0(x)+1)}\\
&\geq \left(\frac{3}4-\frac{3}{4\deg_0(x)}\right)c_0(x).
\end{align*}
	
Now, recall that for all $v'\in V(\mathrm{LT}(G))\setminus V(G)$, we have $c_1(v')\geq c_0(v)$. Thus,
	\begin{align*}
	C(\mathrm{LT}(G))&=\cfrac{\sum_{v\in V(\mathrm{LT}(G))\cap V(G)}c_1(v)+\sum_{v'\in V(\mathrm{LT}(G))\setminus V(G)}c_1(v')}{2|V(G)|}\\
	&\geq \frac{\sum_{v\in V(G)}\left( \displaystyle\frac{3}4- \displaystyle\frac{3}{4\deg_0(v)}\right)c_0(v)+\sum_{v\in V(G)}c_0(v)}{2|V(G)|}\\
	&\geq \frac{\left(\displaystyle \frac{7}4-\frac{3}{4\delta(G)}\right)\sum_{v\in V(G)}c_0(v)}{2|V(G)|}\\
	&=\left(\frac{7}8-\frac{3}{8\delta(G)}\right)C(G),
	\end{align*}
completing the proof of Lemma~\ref{lemma transitive clustering change}.
\end{proof}

With the preceding lemma, we can improve the lower bound on Theorem~\ref{theorem ILT clustering} to derive a slightly better bound for the clustering coefficient of ILT.

\begin{corollary}\label{corollary transitive clustering coefficient}
	If $G$ is a graph, then
	\[
	C(\ilt{t}G)=\Omega\left(\left(\frac{7}8\right)^t t^{-3/7}\right).
	\]
\end{corollary}

\begin{proof}
	Let $\delta_t=\delta(\ilt{t}G)$. Note that $\delta_t=\delta_{t-1}+1$, so $\delta_t\geq t$. We may iteratively apply Lemma~\ref{lemma transitive clustering change} to obtain that
	\[
	C({\ilt{t}G})\geq C(G)\prod_{i=1}^t\left(\frac{7}8-\frac{3}{8\delta_i}\right)\geq C(G)\prod_{i=1}^t\left(\frac{7}8-\frac{3}{8i}\right).
	\]
Observe that
	\begin{align*}
	\prod_{i=1}^t\left(\frac{7}8-\frac{3}{8i}\right)&=\left(\frac{7}8\right)^t \frac{1}{t!} \prod_{i=1}^t\left(i-\frac{3}7\right)\\
	&=\left(\frac{7}8\right)^t \frac{1}{t!} \prod_{i=1}^{t}\left(t+\frac{4}7-i\right)\\
	&=\left(\frac{7}8\right)^t \frac{\Gamma(t+4/7)}{t!\Gamma(4/7)}\\
	&=\left(\frac{7}8\right)^t \frac{\Gamma(t+4/7)}{t\Gamma(t)  \Gamma(4/7)}=\Theta\left(\left(\frac{7}8\right)^tt^{-3/7}\right),
	\end{align*}
	where $\Gamma$ is the well-known Gamma function. Thus,
	\[
	C(\ilt{t}G)=\Omega\left(\left(\frac{7}8\right)^tt^{-3/7}\right),
	\]
	completing the proof of Corollary~\ref{corollary transitive clustering coefficient}.
\end{proof}

Towards bounding the clustering coefficient for sequences with bounded gaps, we show that the clustering coefficient is bounded away from $0$ whenever we perform an anti-transitive step.

\begin{lemma}\label{lemma anti-transitive clustering bound}
Let $G$ be a graph and $S=(s_0,s_1,s_2,\dots)$ be a binary sequence with bounded gaps between 0's. Let $k$ be the constant such that there is no gap of length $k$, and let $\tau$ be the third index such that $s_{\tau}=0$. For all $t\geq \tau$, if $s_t=0$, then
\[
C(\ilm{t}SG)\geq (1+o(1))\frac{1}{2^{2k+4}}.
\]
\end{lemma}

\begin{proof}
Let $G_t=\ilm{t}SG$, $t\geq \tau$, and let $\beta_1=\beta_1(t)$, and $\beta_2=\beta_2(t)$ be the two largest indices such that $\beta_1<\beta_2\leq t$, and $s_{\beta_1}=s_{\beta_2}=0$. We then have that $\beta_2\geq t-k$ and $\beta_1\geq \beta_2-k\geq t-2k$.

If $v\in V(G_{\beta_1-1})$, then $\deg_{\beta_1}(v)=\frac{n_{\beta_1}}{2}-1$. If $s_{\beta_1+1}=1$, then we have that
\[
\deg_{\beta_1+1}(v)=2\deg_{\beta_1}(v)+1=\frac{n_{\beta_1+1}}{2}-1.
\]
Since all entries between $\beta_1$ an $\beta_2$ are 1, we have, inductively that
\[
\deg_{\beta_2-1}(v)=\frac{n_{\beta_2-1}}{2}-1.
\]
Thus, $v$ is adjacent to exactly half of the vertices in $V(G_{\beta_2})\setminus V(G_{\beta_2-1})$.

Let $X_{\beta_2}=N_{\beta_2}(v)\cap (V(G_{\beta_2})\setminus V(G_{\beta_2-1}))$ and $Y_{\beta_2}=(V(G_{\beta_2})\setminus V(G_{\beta_2-1}))\setminus N_{\beta_2}(v)$. Therefore, we find that $|X_{\beta_2}|=|Y_{\beta_2}|=n_{\beta_2}/4$.
	
If $s_{\beta_2+1}=1$, then we define $X_{\beta_2+1}$ to be the vertices in $X_{\beta_2}$ and the clones of vertices in $X_{\beta_2}$ born at time $\beta_2+1$. Similarly, let $Y_{\beta_2+1}$ be the vertices in $Y_{\beta_2}$ along the clones of $Y_{\beta_2}$ born at time $\beta_2+1$. Note that there are no edges between $X_{\beta_2+1}$ and $Y_{\beta_2+1}$ and $|X_{\beta_2+1}|=|Y_{\beta_2+1}|=n_{\beta_2+1}/4$. In addition, $X_{\beta_2+1}\subseteq N_{\beta_2+1}(v)$, while $Y_{\beta_2+1}\cap N_{\beta_2+1}(v)=\emptyset$.

Inductively, we can continue in this fashion until we have sets $X_{t-1}$ and $Y_{t-1}$ of size $n_{t-1}/4$ with no edges between them and $X_{t-1}\in N_{t-1}(v)$, while $Y_{t-1}\cap N_{t-1}(v)=\emptyset$. After an anti-transitive step, the vertices in $X_{t-1}$ and $v$ will be adjacent to every clone of a vertex in $Y_{t-1}$. We then have that $N_t(v)$ contains at least $|X_{t-1}|\cdot |Y_{t-1}|=n_t^2/64$ edges. Since $\deg_t(v)=\frac{n_t}{2}-1$, we have that
	\[
	c_t(v)\geq \frac{n_t^2/64}{\displaystyle \binom{n_t/2-1}2}=(1+o(1))\frac{1}8.
	\]
Note this is holds for all vertices $v\in V(G_{\beta_1-1})$. There are $n_{\beta_1-1}\geq n_{t-2k-1}=\frac{n_t}{2^{2k+1}}$ such vertices, and hence,
	\[
	C(G_t)\geq \frac{(1+o(1)){\displaystyle \frac{1}8 \cdot \displaystyle \frac{n_t}{2^{2k+1}}}}{n_t}=(1+o(1))\frac{1}{2^{2k+4}},
	\]
completing the proof of Lemma~\ref{lemma anti-transitive clustering bound}
\end{proof}

We now have the tools to prove the main result of this section.

\begin{theorem}\label{theorem clustering coefficient}
	Let $G$ be a graph, $S=(s_0,s_1,s_2,\dots)$ be a binary sequence with bounded gaps between zeroes, and $k$ be an absolute constant such that there is no gap of length $k$. If $\tau$ is the third index such that $s_{\tau}=0$, then for all $t\geq \tau$, the clustering coefficient
	\[
	C(\ilm{t}S{G})\geq (1+o(1))\left(\frac{7}8\right)^k\frac{1}{4^{k+2}}.
	\]
\end{theorem}

\begin{proof}
Let $G_t=\ilm{t}SG$. Let $\beta\leq t$ denote the largest index such that $s_{\beta}=0$, and so $\beta\geq t-k$. Note that since at least one anti-transitive step happens before time $\beta$, we have that the maximum degree $\Delta(G_{\beta-1})=n_{\beta-1}/2-1$. Thus, after the anti-transitive step at time $\beta$, we have minimum degree $\delta(G_{\beta})\geq n_{\beta-1}/2=n_{\beta}/4\geq n_t/2^{k+2}$. Note that the minimum degree does not decrease after a transitive step, so $\delta(G_{t'})\geq n_t/2^{k+2}$ for all $t'\geq \beta$.
	
By Lemma~\ref{lemma anti-transitive clustering bound}, we have that $C(G_{\beta})\geq (1+o(1))/4^{k+2}$. By repeated application of Lemma~\ref{lemma transitive clustering change}, we have that
	\begin{align*}
	C(G_t)&\geq \left(\frac{7}8-\frac{3\cdot 2^{k+2}}{8n_t}\right)^{t-\beta} C(G_{\beta})\\
	&\geq (1+o(1))\left(\frac{7}8\right)^k\frac{1}{4^{k+2}},
	\end{align*}
completing the proof of Theorem~\ref{theorem clustering coefficient}.
\end{proof}

\section{Graph Parameters}\label{GPsec}

In this section, we will explore a number of classical graph parameters for ILM graphs, including the chromatic number, domination number, and diameter. For a more detailed discussion of graph parameters, the reader is pointed to \cite{west}.

A shortest path between two vertices is called a $\emph{geodesic}$. The diameter of a graph is the length of the longest geodesic; that is, it is the furthest distance between any two vertices. The radius of a graph is the largest integer $r$ such that every vertex has at least one vertex at distance $r$ from it.

\begin{lemma}\label{lemma chi+1}
If $G$ is a graph, then $\chi(\mathrm{LT}(G))=\chi(G)+1$. If the radius of $G$ is at least $3$, then $\chi(\mathrm{LAT}(G))=\chi(G)+1$.
\end{lemma}

\begin{proof}
Observe first that after a transitive step or an anti-transitive step, the vertex set of resulting graph can be partitioned into a set which induces $G$ and an independent set. Thus, the chromatic number goes up by at most one.

Let $F$ and $H$ be graphs such that $F$ is an induced subgraph of $H$. We claim that if for every vertex $v\in V(F)$, there exists a vertex $u\in V(H)$ with $N[v]\cap V(F)\subseteq N(u)$, then $\chi(H)>\chi(G)$.

Now, assume to the contrary that $\chi(\mathrm{LT}(G))=\chi(G).$ This induces a proper coloring on $G$. It is well-known and straightforward to see that if $G$ is properly colored with $\chi(G)$ colors, then there must be a vertex $v\in V(G)$ such that $N[v]\cap V(G)$ has a vertex of every color in it. Hence, there is no possible color for the clone $v'$, a contradiction.

Similarly, in $\mathrm{LAT}(G)$, for every vertex $v\in V(G)$, if $G$ has radius at least $3$, there exists a vertex $u\in V(G)$ at distance $3$ from $v$, so the anti-transitive clone $u^*\in V(\mathrm{LAT}(G))$ has the property that $N[v]\cap V(G)\subseteq N(v^*)$. Thus, by the same argument in the preceding paragraph, $\chi(\mathrm{LAT}(G))=\chi(G)+1$.
	\end{proof}

We prove that after one anti-transitive step, the radius is at least $3$, and transitive steps preserve the radius.

\begin{lemma}\label{lemma radius 3}
If $G$ is any graph, then the radius of $\mathrm{LAT}(G)$ is at least three. Further, if $G$ is a graph with radius at least $3$, then the radius of $\mathrm{LT}(G)$ is at least 3.
\end{lemma}

\begin{proof}
To see the first part of the lemma, note that if $x\in V(G)$, and $x^*\in V(\mathrm{LAT}(G))\setminus V(G)$ is the anti-transitive clone of $x$, then $N[x]\cap N[x^*]=\emptyset$, so $d(x,x^*)\geq 3$.
	
For the second part of the lemma, first, note that if $u,v\in V(G)$, then the distance between $u$ and $v$ in $G$ is the same as the distance in $\mathrm{LT}(G)$. Indeed, since for any vertex $w\in V(G)$, and its transitive clone $w'$, we have $N(w')=N[w]\cap V(G)$, no $u$-$v$-geodesic will contain both $w$ and $w'$, and furthermore, for any $u$-$v$-geodesic containing $w'$, we can replace $w'$ with $w$, giving us another $u$-$v$-geodesic. We then have that there is a $u$-$v$-geodesic using only vertices in $V(G)$, so the distance between $u$ and $v$ is the same in both $G$ and $\mathrm{LT}(G)$, so every vertex in $G$ has a vertex at distance $3$. Note that if $d(u,v)\geq 3$ for $u,v\in V(G)$, then for the transitive clone, $v'$ of $v$, we also have $d(u,v')=d(u,v)\geq 3$, proving the proposition.
\end{proof}

We are now ready to prove Theorem~\ref{thm:chrom}.

\begin{theorem}\label{thm:chrom}	
If $G$ is a graph and $S = (s_0, s_1, s_2, \ldots)$ is a binary sequence, then the chromatic number satisfies the following:
	
	\[
	\chi(G) + t -1 \leq \chi({\ilm{t}SG}) \leq {\chi}(G) + t,
	\]
for all $t \geq 0$.
\end{theorem}

\begin{proof}
For the upper bound, note that every transitive or anti-transitive step introduces an independent set, so using a new color for each time-step, we achieve a proper coloring of $\ilm{t}SG$ with $\chi(G)+t$ colors.
	
For the lower bound, first note that if $S$ is the all-1's sequence, then by repeated application of Lemma~\ref{lemma chi+1}, we have $\chi({\ilt{t}G})=\chi(G)+t$, so we are done. Let $\tau$ be the first index such that $s_{\tau}=0$. We find that $\ilm{\tau-1}SG=\ilt{\tau-1}G$, so $\chi({\ilm{\tau-1}SG})=\chi(G)+\tau$. It is straightforward to see that
\[
\chi({\ilm{\tau}SG})\geq \chi({\ilm{\tau-1}SG})=\chi(G)+\tau-1.
\]

By repeated application of Lemma~\ref{lemma radius 3}, $\ilm{t}SG$ has radius at least $3$ for all $t\geq \tau$, so by repeated application of Lemma~\ref{lemma chi+1}, $\chi({\ilm{t}SG})=\chi({\ilm{\tau}SG})+t-\tau\geq \chi(G)+\tau-1+t-\tau=\chi(G)+t-1$, completing the proof.
\end{proof}

A \emph{dominating set} in a graph $G$ is a set $D \subseteq V(G)$ such that $V(G) \subseteq N[D]$. The \emph{domination number}, denoted by $\gamma(G)$, is the minimum size of such a set. Given two vertices $u,v\in V(G)$ we will say \emph{the closed neighborhoods of $u$ and $v$ partition the vertex set of $G$}, or simply \emph{$u$ and $v$ partition the vertex set of $G$}, if $N[u]\cup N[v]=V(G)$ and $N[u]\cap N[v]=\emptyset$.

We show that in ILM graphs with at least two bits equal to zero, the domination number will be at most 3 after a sufficient number of time-steps.  However, for ILT graphs a straightforward discussion will show that transitive steps preserve the domination number, or more specifically for a graph $G$ we have $\gamma(G)=\gamma(\ilt{t}G)$ for all $t\geq 0.$

Note that if $A\subseteq V(G)$ is a dominating set in $G$, then $A$ will also dominate in $\ilt{t}G$, so $\gamma(G)\geq\gamma(\ilt{t}G)$. Assume $B$ is a dominating set of $\mathrm{LT}(G)$. If $B\subseteq V(G)$, then $B$ also dominates in $G$, so $\gamma(G)\leq \gamma(\mathrm{LT}(G))$. Otherwise, there exists some clone $w'\in V(\mathrm{LT}(G))\setminus V(G)$, say $w'$ is the clone of some vertex $w\in V(G)$, such that $w'\in B$. Note that $(B\setminus\{w'\})\cup \{w\}$ is a dominating set since $N[w']\subseteq N[w]$, which implies that we can always find a dominating set of $\mathrm{LT}(G)$ in $V(G)$, and thus, $\gamma(G)\leq \gamma(\mathrm{LT}(G))$. By induction, this implies that $\gamma(G)\leq \gamma(\ilt{t}G)$, and we are done.

We have the following theorem on the domination number of general ILM graphs.

\begin{theorem}\label{theorem dom3}
Let $G$ be a graph and $S=(s_0,s_1,s_2,\dots)$ be an binary sequence with at least two bits equal to $0$. If $\tau$ is the second index such that $s_{\tau}=0$, then for all $t\geq \tau+1$, $\gamma(\ilm{t}SG) \leq 3.$
\end{theorem}

\begin{proof}
Let $G_t=\ilm{t}SG$, and let $\beta_1<\beta_2<t$ be the largest and second largest indices such that $s_{\beta_1}=s_{\beta_2}=0$. Let $v$ be any vertex in $V(G_{\beta_1})$. Since $s_{\beta_1}=0$, $G_{\beta_1+1}=\mathrm{LAT}(G_{\beta_1})$. Let $v^*\in V(G_{\beta_1}+1)$ be the anti-transitive clone of $v$.  and let $v^{**}\in V(G_{\beta_2}+1)$ denote the anti-transitive clone of $v^*$ when we perform the anti-transitive step at time $\beta_2$. We claim that $D=\{v,v^*,v^{**}\}$ is a dominating set of $G_{\beta_2+1}$.
	
Any vertex in $V(G_{\beta_2})$ that is not adjacent to $v^*$ is adjacent to $v^{**}$ by definition, so we need only focus on the new vertices in $V(G_{\beta_2+1})\setminus V(G_{\beta_2})$. Note that $N_{\beta_1+1}[v]\cap N_{\beta_1+1}[v^*]=\emptyset$ by the definition of an anti-transitive step. Since there are only transitive steps between time $\beta_1$ and $\beta_2$, we also have that $N_{\beta_2}[v]\cap N_{\beta_2}[v^*]=\emptyset$. Thus, any vertex $y\in V(G_{\beta_2})$ is adjacent to at most one of $v$ or $v^*$, and so any newly created vertex $z\in V(G_{\beta_2+1})\setminus V(G_{\beta_2})$ is adjacent to at least one of $v$ or $v^*$. Thus, $D$ is a dominating set of $G_{\beta_2+1}$. Then there are only transitive steps between time $\beta_2$ and time $t$. The result follows from the fact that ILT steps preserve the domination number.
\end{proof}

The preceding theorem tells us that the domination number of an ILM graph after two anti-transitive steps is either $1$, $2$, or $3$. After an anti-transitive step, we cannot have a dominating vertex, so that leaves the possibility of domination number $2$ or $3$. We can characterize exactly when the domination number is $2$ and when it is $3$. To do so, first, we establish a helpful lemma.

\begin{lemma}\label{lemma domination pair partitions the vertex set}
Let $G$ be a graph. If $G$ contains a pair of vertices whose closed neighborhoods partition the vertex set, then the same pair also partition the vertex set in $\mathrm{LAT}(G)$ and $\mathrm{LT}(G)$. Further, if $\mathrm{LT}(G)$ contains a pair of vertices whose closed neighborhoods partition the vertex set, then $G$ does as well.
\end{lemma}

\begin{proof}
First, let $u,v\in V(G)$ be a pair of vertices that partition the vertex set of $G$. In $\mathrm{LAT}(G)$, every anti-transitive clone of the vertices in $N[u]\cap V(G)$ is adjacent to $v$ and each clone of a vertex in $N[v]\cap V(G)$ is adjacent to $u$, so $N[u]\cup N[v]=V(\mathrm{LAT}(G))$. It is straightforward to see that $N[u]\cap N[v]=\emptyset$. so $u$ and $v$ partition the vertex set of $\mathrm{LAT}(G)$.
	
Similarly, we claim that $u$ and $v$ will partition the vertex set of $\mathrm{LT}(G)$. Indeed, every transitive clone of a vertex in $N[u]\cap V(G)$ is adjacent to $u$, and every clone of a vertex in $N[v]\cap V(G)$ is adjacent to $v$, and it follows that $N[u]\cap N[v]=\emptyset$.
	
Let $x$ and $y$ be vertices that partition the vertex set of $\mathrm{LT}(G)$. If $x,y\in V(G)$, then they partition the vertex set of $G$. If not, then there exist a vertex $w\in V(G)$ such that the transitive clone of $w$, $w'=x$ without loss of generality. If there also exists a vertex $z\in V(G)$ such that the transitive clone of $z$, $z'=y$, then $x$ and $y$ must be the only clones so $|V(G)|=2$. Subsequently, if $G=K_2$, then $\mathrm{LT}(G)$ does not have a pair of vertices that partition the vertex set, thus a contradiction. Similarly, if $G=2K_1$, then $G$ does have a pair of vertices that partition the vertex set, and again a contradiction. Thus, we may assume that $x=w'$ and $y\in V(G)$. In fact, the vertex $y$ must be adjacent to every vertex in $(V(\mathrm{LT}(G))\setminus V(G))\setminus \{x\}$, and consequently, $y$ must be adjacent to every vertex in $V(G)\setminus \{w\}$. We find that $w$ must be an isolated vertex in $G$, since otherwise, $N[x]\cap N[y]\neq \emptyset$. Hence, $w$ and $x$ partition the vertex set of $G$, and we are done.
\end{proof}

We next characterize exactly which graphs and sequences give rise to an ILM graph with domination number $2$.

\begin{theorem}\label{theorem if domination=2}
Let $G$ be a graph and let $S=(s_0,s_1,s_2\dots)$ be any binary sequence that contains at least one bit equal to $0$. If $\tau_1$ is the first index such that $s_{\tau_1}=0$, then for all $t\geq \tau_1+1$, $\gamma(\ilm{t}SG)=2$ if and only if one of the following statements holds.
\begin{enumerate}
\item The graph $G$ has a pair of vertices whose closed neighborhoods partition the vertex set.
\item The graph $G$ contains an isolated vertex and $\tau_1=0$.
\item The graph $G$ contains a dominating vertex, $s_{\tau_1}=s_{\tau_1+1}=0$, and $t\geq \tau_1+2$.
\end{enumerate}
\end{theorem}
	
\begin{proof}
Let $G_t=\ilm{t}SG$ for all $t\geq 0$. First let us prove the reverse direction. If $G$ has a pair of vertices whose closed neighborhoods partition the vertex set, then so does $G_t$ by Lemma~\ref{lemma domination pair partitions the vertex set}, so $\gamma(G_t)=2$. If $G$ contains an isolated vertex, say $v$ and $\tau_1=0$, then $v$ and its anti-transitive clone, $v^*\in V(G_1)$ form a pair of vertices whose closed neighborhoods partition the vertex set. Indeed, $v^*$ will be adjacent to every vertex in $V(G)\setminus\{v\}$, and $v$ will be adjacent to every vertex in $(V(G_1))\setminus V(G))\setminus\{v^*\}$, and furthermore, we have that $N[v]\cap N[v']=\emptyset$. Thus, by Lemma~\ref{lemma domination pair partitions the vertex set}, in this case $\gamma(G_t)=2$ for all $t\geq \tau_1+1=1$.
	
Finally, if $G$ contains a dominating vertex, $s_{\tau_1}=s_{\tau_1+1}=0$, and $t\geq \tau_1+2$, first note that $G_{\tau_1}=\ilt{\tau_1}G$ still has a dominating vertex, say $v$. After performing an anti-transitive step, the anti-transitive clone $v^*$ of $v$ will be isolated, so by the preceding argument, after the second anti-transitive step (at time $\tau_1+1$), this will give us a graph with a pair of vertices whose closed neighborhoods partition the vertex set, so we are done.
	
We prove the forward direction. First note that $\Delta({G_{\tau_1+1}})=n_{\tau_1+1}/2-1$ since $G_{\tau_1+1}$ was the result of an anti-transitive step. Note that if the next step is transitive, a vertex of degree $n_{\tau_1+1}/2-1$ will gain $n_{\tau_1+1}/2$ new neighbors, which gives it degree $n_{\tau_1+1}-1=n_{\tau_1+2}/2-1$, and if the next step is anti-transitive, such a vertex will also gain $n_{\tau_1+1}/2$ new neighbors, so in all cases, $\Delta(G_{\tau_1+2})=n_{\tau_1+2}/2-1$. In fact, by the same reasoning, for all $t\geq \tau_1+1$, $\Delta(G_t)=n_t/2-1$. Thus, if $\gamma(G_t)=2$, $G_t$ must have a pair of vertices whose closed neighborhoods partition the vertex set.
	
If $G$ has a pair of vertices whose closed neighborhoods partition the vertex set, then we are done, so assume otherwise. Let $0<\tau_2\leq t$ be the smallest index such that $G_{\tau_2}$ has a pair of vertices whose closed neighborhoods partition the vertex set, say vertices $u,v\in V(G_{\tau_2})$. As a consequence of Lemma~\ref{lemma domination pair partitions the vertex set}, we must then have that $s_{\tau_2-1}=0$. If both $u$ and $v$ were in $V(G_{\tau_2-1})$, then they would be a pair of vertices that partition the vertex set, which contradicts the choice of $\tau_2$, so at least one, say $u$ is in $V(G_{\tau_2})\setminus V(G_{\tau_2-1})$. If we also have that $v\in V(G_{\tau_2})\setminus V(G_{\tau_2-1})$, then since $u$ and $v$ dominate $G_{\tau_2}$,  but are not adjacent to any other vertices in $V(G_{\tau_2})\setminus V(G_{\tau_2-1})$, we must have that $V(G_{\tau_2})\setminus V(G_{\tau_2-1})=2$, so $n_{\tau_2}=4$, so $G_{\tau_2-1}$ is either $K_2$ or $2K_1$. $2K_1$ has a pair of vertices  whose closed neighborhoods partition the vertex set, and in $\mathrm{LT}(K_2)$, the two new clones are isolated vertices, so they do not partition the vertex set. Thus, we have that $u\in V(G_{\tau_2})\setminus V(G_{\tau_2-1})$ and $v\in V(G_{\tau_2-1})$.
	
Let $v^*\in V(G_{\tau_2})\setminus V(G_{\tau_2-1})$ be the anti-transitive clone of $v$. Since $v$ is not adjacent to $v^*$, we must have that $v^*=u$, since otherwise, $v^*$ would not be dominated in $G_{\tau_2}$. But then, $v$ must be adjacent to all the vertices in $V(G_{\tau_2})\setminus V(G_{\tau_2-1})\setminus\{u\}$, so $v$ is not adjacent to any vertex in $V(G_{\tau_2-1})\setminus\{v\}$, meaning that $G_{\tau_2-1}$ has an isolated vertex. If $\tau_2-1=0$, then $G$ has an isolated vertex, and we are done, so assume that $\tau_2-1\geq 1$. We must then have that $s_{\tau_2-2}=0,$ since the minimum degree after a transitive step is at least $1$. We claim that either $G_{\tau_2-2}=K_1$, or $v\in V(G_{\tau_2-1})\setminus V(G_{\tau_2-2})$. Indeed, if $G_{\tau_2-2}\neq K_1$ and $v\in V(G_{\tau_2-2})$, then $v$ is isolated in $V(G_{\tau_2-2})$, but this would imply that in $G_{\tau_2-1}$, $v$ would be adjacent to the clones of the vertices in $V(G_{\tau_2-2})\setminus\{v\}$, giving a contradiction. In addition, if $G_{\tau_2-2}=K_1$, we must have that $\tau_2-2=\tau_1=0$, so $G=K_1$ has a dominating vertex, $s_0=s_1=0$, and $t\geq \tau_2\geq \tau_1+2$, so we are done. Thus, we can assume that $v\in V(G_{\tau_2-1})\setminus V(G_{\tau_2-2})$.
	
Let $x\in V(G_{\tau_2-2})$ be such that $v$ is the anti-transitive clone of $x$. The vertex $x$ must then be a dominating vertex in $G_{\tau_2-2}$ since otherwise, $v$ would have neighbors in $G_{\tau_2-1}$. Since a dominating vertex cannot result from an anti-transitive step, we must have that $s_i=1$ for all $i<\tau_2-2$, which implies that $\tau_1=\tau_2-2$, and thus, $s_{\tau_1}$ and $s_{\tau_1+1}=0$ and $t\geq \tau_1+2$. Finally, since $G_{\tau_2-2}$ contains a dominating vertex, $G$ must as well, completing the proof.
\end{proof}

 We now move our attention to the diameter of ILM graphs. To determine the diameter, we first need to establish a few axillary results.

\begin{lemma}\label{lemma antitransitive disconnected}
Let $G$ be a graph. The graph $\mathrm{LAT}(G)$ is disconnected if and only if $G$ has a dominating vertex, or is the disjoint union of two complete graphs.
\end{lemma}

\begin{proof}
The reverse direction is straightforward to verify, so we will focus on the forward direction. We will assume that $G$ has no dominating vertex and is not the disjoint union of two cliques, and show that $\mathrm{LAT}(G)$ is connected. Let $G$ have components $A_1,A_2,\dots,A_\ell$, and let $A_i^*\subseteq V(\mathrm{LAT}(G))\setminus V(G)$ denote the clones of the vertices in $A_i$ for each $1\leq i\leq \ell$. Every vertex in $A_i$ is adjacent to every vertex in $A_j^*$ whenever $i\neq j$. Whenever $\ell\geq 3$, $\mathrm{LAT}(G)$ is connected, so it must be the case that $\ell\in \{1,2\}$.

If $\ell=2$, then all the vertices in $A_1\cup A_2^*$ must be connected, and similarly all vertices in $A_1^*\cup A_2$ must be connected. Since $G$ is not a disjoint union of cliques, there is a pair, say $u,v\in A_1$ such that $uv\not\in E(G)$. Thus, $u\in A_1$ is adjacent to $v^*\in A_1^*$, so all the vertices, are connected.

If $\ell=1$ and $G$ does not contain a dominating vertex, then it is evident that the vertices in $V(G)$ are all connected. For any $u\in V(G)$, there exists $v\in V(G)$ such that $uv\not\in E(G)$, so $u^*v\in E(\mathrm{LAT}(G))$. Therefore, every vertex in $V(\mathrm{LAT}(G))\setminus V(G)$ is connected to $V(G)$, so $G$ is connected, completing the proof.
\end{proof}

\begin{lemma}\label{lemma non dominating triples diameter}
Let $G$ be a graph. If $\gamma(G)\geq 3$, then $\mathrm{diam}(\mathrm{LAT}(G))\leq 3$.
\end{lemma}

\begin{proof}
Let $x,y\in V(G)$ be a pair of vertices. We will show that in $\mathrm{LAT}(G)$, $d(a,b)\leq 3$ for all $a\neq b$, $a,b\in \{u,v,u^*,v^*\}$, where $u^*,v^*\in V(\mathrm{LAT}(G))\setminus V(G)$ are the anti-transitive clones of $u$ and $v$ respectively. Since $\gamma(G)\geq 3$, there exists some $z\in V(G)$ that is not adjacent to either $x$ or $y$. if $z^*$ is the anti-transitive clone of $z$, then $(x,z^*,y)$ is a path in $\mathrm{LAT}(G)$, so $d(x,y)\leq 2$. A similar argument shows $d(x,z)=d(y,z)=2$ in $\mathrm{LAT}(G)$. Since $x^*z$ and $y^*z$ are edges in $\mathrm{LAT}(G)$, we have that $d(x^*,y^*)=2$, and $d(a,b)\leq 3$ for all $a\in \{x^*,y^*\}$ and $b\in \{x,y\}$. Since this is true for all pairs $x$ and $y$ in $V(G)$, $\mathrm{diam}(\mathrm{LAT}(G))\leq 3$.
\end{proof}

Now we will see that the diameter of an ILM graph eventually becomes $3$, usually after only two anti-transitive steps.

\begin{theorem}\label{theorem diameter after two time-steps}
Let $G\neq K_1$ be a graph that is not the disjoint union of two cliques, and  $S=(s_0,s_1,s_2,\dots)$ be a binary sequence with at least two bits equal to $0$. If $\tau$ is the index of the second time-step such that $s_{\tau}=0$, then for all $t\geq \tau+1$, $\mathrm{diam}(\ilm{t}SG)=3$.
\end{theorem}

\begin{proof}
Let $G_t=\ilm{t}SG$. By repeated application of Lemma~\ref{lemma radius 3}, we have that the radius of $G_t$ is at least three for all $t\geq \tau+1$, so the diameter of $G_t$ is at least three for all $t\geq \tau+1$.
	
We first consider the upper bound. Since transitive steps preserve the diameter of a graph, we need only focus on the diameter immediately after the most recent anti-transitive step, so without loss of generality we will assume that $s_{t-1}=0$, so that $G_t=\mathrm{LAT}(G_{t-1})$.

Observe that for any graph $F\not\in \{K_1,2K_1\}$, we have that neither $\mathrm{LT}(F)$ nor $\mathrm{LAT}(F)$ are a disjoint union of two cliques. Indeed, first if $F=K_2$, then $\mathrm{LT}(F)$ is connected and $\mathrm{LAT}(F)$ has three components, so the claim is true for $F=K_2$. Now let $|V(F)|\geq 3$. If $\mathrm{LT}(F)$ or $\mathrm{LAT}(F)$ is a disjoint union of cliques, then no two clones can be in the same component, implying that $\mathrm{LT}(F)$ or $\mathrm{LAT}(F)$ must have at least $|V(F)|>2$ components, and thus, cannot be a disjoint union of two cliques. We then have that $G_i$ is not a disjoint union of two cliques for any $i\geq 0$, and furthermore via induction, it suffices to proof the theorem in the case where $t-1=\tau$ is the second index with $s_{t-1}=0$, and where $s_0=0$.
 	
By Lemma~\ref{lemma non dominating triples diameter}, we are done if $\gamma(G_{t-1})\geq 3$. Since there was one anti-transitive step before time $t-1$, $\Delta(G_{t-1})=\frac{n_{t-1}}2-1$, so $\gamma(G_{t-1})\neq 1$, so the only remaining case is when $\gamma(G_{t-1})=2$, implying that $G_{t-1}$ has a pair of vertices whose closed neighborhoods partition the vertex set.
	
We will next show that $G_{t-1}$ and $G_t$ must be connected. By theorem \ref{theorem if domination=2}, since there is only one $0$ preceding $s_{t-1}$ in $S$, $G$ cannot have a dominating vertex if $G_{t-1}$ has domination number $2$. By Lemma~\ref{lemma antitransitive disconnected}, $G_1$ is connected, and $G_{t-1}$ is as well since transitive steps preserve connectedness. Similarly, $G_{t-1}$ cannot have a dominating vertex since $\Delta(G_{t-1})=n_{t-1}/2-1$, and also is not a disjoint union of two cliques, so again by Lemma~\ref{lemma antitransitive disconnected} $G_t$ is connected.
	
Let $x$ and $y$ be a pair of vertices that partition the vertex set of $G_{t-1}$. By Lemma~\ref{lemma domination pair partitions the vertex set}, it follows that $x$ and $y$ also partition the vertex set of $G_t$. Hence, every pair of vertices in $N_t[x]$ are at distance at most two from each other, and similarly for $N_t[y]$. If a vertex $u\in N_t[x]$ has a neighbor in $N_t[y]$, or vice versa, then $u$ is distance at most $3$ from every vertex. Thus, we need only worry about pairs $u,v\in V(G_t)$ with $N_t[u]\subseteq N_t[x]$ and $N_t[v]\subseteq N_t[y]$.
	
If both $u$ and $v$ are in $V(G_{t-1})$, then $\deg_t(u)=\deg_t(v)=n_t/2-1$, and so by connectivity, $d(u,v)\leq 3$. Thus, one of the vertices must be a new anti-transitive clone, say $v$ is the clone of some vertex $w\in N_{t-1}[x]$. We then have that $N_{t-1}[w]=N_{t-1}[x]$; otherwise, $v$ would be adjacent to some vertex in $N_{t-1}[x]$, and since $w$ and $y$ partition the vertex set of $G_{t-1}$, we can assume, without loss of generality,  that actually $w=x$. We then have that $N_{t-1}[y]\subseteq N_t[v]$ since $N_{t-1}[x]\cap N_{t-1}[y]=\emptyset$. Since $G_{t-1}$ is connected, there is an edge from $N_{t-1}[x]$ to $N_{t-1}[y]$, so if $N_t[u]=N_t[x]$, then this gives a path from $u$ to $v$ of length $3$. Thus, $N_t[u]\subsetneq N_t[x]$, and thus, $\deg_t(u)<n_t/2-1$, so $u$ must also be an anti-transitive clone of some vertex in $G_{t-1}$, and by a symmetric argument to the one for $v$, $u$ must be the clone of a vertex that dominates $N_{t-1}[y]$. Assume, without loss of generality, that $u$ is the clone of $y$. Since $u$ and $v$ are clones of a pair of vertices that partition the vertex set of $G_{t-1}$, $N_t[u]=V(G_{t-1})\setminus N_t[y]=N_{t-1}[x]$ and $N_t[v]=N_{t-1}[y]$. Since $G_{t-1}$ is connected, there is a path of length $3$ from $u$ to $v$, and we are done.
\end{proof}

For the forbidden graphs in Theorem~\ref{theorem diameter after two time-steps}, it can take up to five anti-transitive steps for $K_1$ to become diameter $3$, since the sequences of graphs $K_1$, $2K_1$, $2K_2$, $2(K_4-e)$ are the graphs one gets from performing the first three anti-transitive steps, and by direct checking, the fourth step gives a graph with diameter $4$.

For $2K_1$, it takes four anti-transitive steps, and then for all other $K_s\cup K_t$, it takes three anti-transitive steps. Indeed, for $K_s\cup K_t$, the first anti-transitive step gives a graph that contains two components, one isomorphic to a complete join between $K_s$ and $\overline{K_t}$, and the other that is the complete join of $K_t$ and $\overline{K_s}$. Since we are not allowing $s=t=1$, one of these components is missing at least one edge, so the next anti-transitive step will yield connectivity. So, after performing any number of transitive steps, once we finally do the second anti-transitive step, the clones of vertices that dominated their components will be at distance $4$ from each other, and then finally after any further anti-transitive step, the  diameter will drop to $3$ and stay there.

\section{Spectral Expansion}\label{SEsec}

For a graph $G$ and sets of vertices $X,Y \subseteq V(G)$, define $E(X,Y)$ to be the set of edges in $G$ with one endpoint in $X$ and the other in $Y.$ For simplicity, we write $E(X)=E(X,X).$ Let $A$ denote the adjacency matrix and $D$ denote the diagonal degree matrix of a graph $G$. The \emph{normalized Laplacian} of $G$ is
\[ \mathcal{L} = I - D^{-1/2}AD^{-1/2}.\]
Let $0 = \lambda_0 \leq \lambda_1 \leq \cdots \leq \lambda_{n-1} \leq 2$ denote
the eigenvalues of $\mathcal{L}$. The \emph{spectral gap} of the normalized Laplacian is defined as
\[
\lambda = \max\{ |\lambda_1 - 1|, |\lambda_{n-1} - 1| \}.
\]

We will use the expander mixing lemma for the normalized Laplacian~\cite{sgt}. For sets of vertices $X$ and $Y$, we use the notation $\vol(X) = \sum_{v \in X} \deg(v)$ for the volume of $X$, $\overline{X} = V \setminus X$ for the complement of $X$, and, $e(X,Y)$ for the number of edges with one end in each of $X$ and $Y.$ Note that $X \cap Y$ need not be empty, and in this case, the edges completely contained in $X\cap Y$ are counted twice. In particular, $e(X,X) = 2 |E(X)|$.

\begin{lemma}[Expander mixing lemma]\cite{sgt}\label{mix}
If $G$ is a graph with spectral gap $\lambda$, then, for all sets $X \subseteq V(G),$
\[
\left| e(X,X) - \frac{(\vol(X))^{2}}{\vol(G)} \right| \leq \lambda \frac{\vol(X)\vol(\overline{X})}{\vol(G)}.
\]
\end{lemma}

A spectral gap bounded away from zero is an indication of bad expansion properties, which is characteristic for social networks, \cite{estrada}. The next theorem represents a drastic departure from the good expansion found in binomial random graphs, where $\lambda = o(1)$~\cite{sgt,CL}.

\begin{theorem}\label{thm:specgap}
Let $G$ be a graph and $S = (s_0, s_1, s_2 \ldots)$ be a binary sequence. For all $t\geq 1$, we have that
\[
\lambda_t \geq \frac{1}4-o(1),
\]
where $\lambda_t$ is the spectral gap of $\ilm{t}SG$.
\end{theorem}

\begin{proof}
Let $G_t=\ilm{t}SG$. Suppose first that $s_{t-1} = 1$, so that $G_t = \mathrm{LT}(G_{t-1})$. Let $X=V(G_t)\setminus V(G_{t-1})$ be the set of cloned vertices added to $G_{t-1}$ to form $G_t$. Since $X$ is an independent set, $e(X,X)=0$. We derive that
	
\begin{align*}
\vol(G_t) &= 6e_{t-1}+2n_{t-1}, \\[0.1cm]
\vol(X) &= 2e_{t-1}+n_{t-1}, \quad \text{and}\\[0.1cm]
\vol(\overline{X}) &= \vol(G_t) - \vol(X) = 4e_{t-1}+n_{t-1}.
\end{align*}
Hence, by Lemma~\ref{mix}, we have that
\begin{align*}
\lambda_t &\ge \frac{(\vol(X))^{2}}{\vol(G_t)} \cdot \frac{\vol(G_t)}{\vol(X)\vol(\overline{X})}\\
&= \frac{\vol(X)}{\vol(\overline{X})} \\
&= \frac{2e_{t-1}+n_{t-1}}{4e_{t-1}+n_{t-1}}>1/2.
\end{align*}
	
Suppose now that $s_{t-1} = 0$, so that $G_t = \mathrm{LAT}(G_{t-1})$. Let $X=V(G_t)\setminus V(G_{t-1})$ be the set of anti-transitive clones added to $G_{t}$ to form $G_{t+1}$. We derive that
	
\begin{align*}
\vol(G_t) &= 2n_{t-1}^2-2e_{t-1}-2n_{t-1}, \\[0.1cm]
\vol(X) &= n_{t-1}^2-2e_{t-1}-n_{t-1}, \quad \text{and}\\[0.1cm]
\vol(\overline{X}) &= \vol(G_t) - \vol(X) = n_{t-1}^2-n_{t-1}.
\end{align*}
Hence, by Lemma~\ref{mix}, we have that
\begin{align*}
\lambda_t &\geq \frac{\vol(X)}{\vol(\overline{X})}\\
&= \frac{n_{t-1}^2-2e_{t-1}-n_{t-1}}{n_{t-1}^2-n_{t-1}}\\
&= 1-\frac{2e_{t-1}}{n_{t-1}^2-n_{t-1}}=\frac{1}4 -o(1),
\end{align*}
where the last equality follows since if $t\geq 2$, $G_{t-1}$ has an independent set of size $n_{t-1}/2$. Therefore, we have that
\[2e_{t-1}\leq 2\binom{n_{t-1}}2-2\binom{n_{t-1}/2}2=\frac{3}4 n_{t-1}^2-\frac{1}2 n_{t-1},
\]	
and the proof follows.
\end{proof}

From the proof of Theorem~\ref{thm:specgap}, it is evident that the constant of $1/4$ is not best possible, and if one had good enough bounds on the implied constants in Theorem~\ref{theorem density}, one could improve the result on the spectral gap as well.

\section{Structural Properties}\label{SRsec}

A graph is \emph{Hamiltonian} if there exists a cycle that visits each node exactly once, and is referred to as a \emph{Hamiltonian cycle}. Let $G$ be a graph with two disjoint cycles of arbitrary length, $C_{(1)}$ and $C_{(2)}$, and let $C_4$ be a $4$-cycle that has one edge in common with $C_{(1)}$, call it $e$, and one with $C_{(2)}$, say $f$. In this case, we call this copy of $C_4$ an \emph{edge switch} between $e$ and $f$, and note that an edge switch implies that there is a cycle that covers all the vertices in $V(C_{(1)})\cup V(C_{(2)})$.

We first note that after relatively few anti-transitive steps, ILM graphs become Hamiltonian.

\begin{theorem}
Let $G\neq K_1$, and $S=(s_0,s_1,s_2,\dots)$ a binary sequence with at least two non-consecutive 0's (that is, if there are exactly two 0's, there is a 1 in between them). If $\tau_1$ and $\tau_2$ are the first two indices such that $s_{\tau_1}=s_{\tau_2}=0$ and $\tau_1\leq\tau_2-2$, then for all $t\geq \tau_2+1$, $\ilm{t}SG$ is Hamiltonian.
\end{theorem}

\begin{proof}
Let $G_t=\ilm{t}SG$. Let $\tau_1$ be the lowest index such that $s_{\tau_1}=0$, and let $\tau_1+2\leq \tau_2\leq t-1$ be the largest index such that $s_{\tau_2}=0$. Note that for all $\ell\geq \tau_1+1$, $\Delta(G_\ell)=\frac{n_\ell}2-1$, so $\delta(\overline{G_{\ell}})=\frac{n_\ell}2$, and since $|V(G_\ell)|\geq 4$, by Dirac's Theorem, $\overline{G_\ell}$ is Hamiltonian. Let $uv\in E(\overline{G_{\tau_2-1}})$ be one of the edges in a Hamiltonian cycle in $\overline{G_{\tau_2-1}}$. We claim that $\overline{G_{\tau_2}}$ has a Hamiltonian cycle with four consecutive vertices that form a clique.
	
Indeed, since in $\overline{G_{\tau_2}}$, the vertices in $V(G_{\tau_2})\setminus V(G_{\tau_2-1})$ form a clique, we can traverse the Hamiltonian path in $\overline{G_{\tau_2-1}}$ from $u$ to $v$, then go to the clone of $v$, and then visit each other new clone in $\overline{G_{\tau_2}}$ before ending at the clone of $u$, and then back to $u$ to close the cycle. In particular, since there are at least $4$ clones, this implies that $\overline{G_{\tau_2}}$ contains a Hamiltonian cycle with four consecutive vertices that form a clique.
	
Let $C=(v_1,v_2,\dots,v_{n_{\tau_2}},v_1)$ be such a Hamiltonian cycle, and say $\overline{G_{\tau_2}}[\{v_1,v_2,v_3,v_4\}]$ is a $K_4$. In $G_{\tau_2+1}$, denote the clone of $v_i$ as $v_i^*$. Note that $v_iv_{i+1}^*$ and $v_i^*v_{i+1}$ are in $E(G_{\tau_2+1})$ for each $1\leq i\leq n_{\tau_2}$ (indices taken modulo $n_{\tau_2}$). The edges of this form give us two disjoint cycles, each of length $n_{\tau_2+1}/2$, one which includes the edge $v_1v_2^*$, and the other which includes the edge $v_3^*v_4$. Note that $v_1v_3^*$ and $v_2^*v_4$ are also edges in $G_{\tau_2+1}$ since $\{v_1,v_2,v_3,v_4\}$ is an independent set in $G_{\tau_2}$. Thus, we have an edge switch between the two cycles of length $n_{\tau_2+1}/2$, so there exists a Hamiltonian cycle.
\end{proof}

The preceding theorem gives us that $\ilat{t}G$ is Hamiltonian for $t\geq 3$ for any graph $G\neq K_1$. Interestingly, while the anti-transitive model becomes Hamiltonian rather quickly, the transitive model can take a long time to become Hamiltonian. Let $G$ be a connected graph, and let $\zeta(G)$ be the smallest integer such that $\ilt{t}G$ is Hamiltonian for all $t\geq \zeta(G)$. Let $\zeta_n=\max\{\zeta(G): |V(G)|=n,G\text{ is connected}\}$. It is not hard to see that $\zeta_1=2$ and $\zeta_2=1$. We will see that $\zeta_n$ is finite, but that it grows with $n$.

\begin{theorem}
For all $n\geq 3$, we have that
\[
\log_2(n-1)\leq \zeta_n \leq \lceil\log_2(n-1)\rceil+1
\]
\end{theorem}

\begin{proof}
For the lower bound, consider the star, $K_{1,n-1}$. Let $A_t$ denote the set of all the vertices in $\ilt{t}{K_{1,n-1}}$ that are descendants of the center vertex in the star (including the center vertex). Note that $A_t$ is a vertex-cut of size $2^t$, whose removal leaves the graph with at least $n-1$ components. Thus, if $2^t<n-1$, then $\ilt{t}{K_{1,n-1}}$ cannot be Hamiltonian, which provides the lower bound.
	
For the upper bound, fix a graph $G$ on $n$ vertices, and let $G_t=\ilt{t}G$. First note that if $G_t$ is Hamiltonian, then $G_{t+1}$ contains a Hamiltonian cycle in which each vertex from $G_t$ is adjacent to its clone in $G_{t+1}$. Indeed, given a Hamiltonian cycle $C=(v_1,v_2,\dots,v_{n_t},v_1)$ in $G_t$, we have the Hamiltonian cycle $C^*=(v_1,v_1',v_2,v_2',\ldots,v_{n_t},v_{n_t}',v_1)$ in $G_{t+1}$, where $v_i'\in V(G_{t+1})\setminus V(G_t)$ is the transitive clone of $v_i$. Thus, to bound $\zeta_n$, it suffices to find the first $t$ such that $G_t$ is Hamiltonian.
	
Before proceeding, we must establish two facts about $\ilt{t}{K_1}$.

\medskip
	
\noindent \textbf{Claim 1:}\label{claim Hcycle in K_1}
For all $t\geq 2$, $\ilt{t}{K_1}$ contains a Hamiltonian cycle in which each vertex of $\ilt{t-1}{K_1}$ is adjacent to its clone in $\ilt{t}{K_1}$.

\medskip

It is straightforward to verify that $\ilt{2}{K_1}$ has such a Hamiltonian cycle, and by an earlier argument, this property is preserved when $t>2$, proving Claim~1.
	
\medskip
	
\noindent \textbf{Claim 2:}\label{claim perfect matching in K_1}
If $V(K_1)=\{v\}$ and $V(\ilt{1}{K_1})=\{u,v\}$, then for all $t\geq 1$, $\ilt{t}{K_1}$ contains a perfect matching that matches the descendants of $u$ to the descendants of $v$ that are not descendants of $u$. For $t\geq 2$, one such matching has the property that the edges can be paired off in such a way that if $e$ and $f$ are paired, there exists two vertices, $x,y\in V(\ilt{t-1}{K_1})$ such that $e=xy'$ and $f=x'y$. We refer to these as a \emph{paired matching}.

\medskip	

Claim~2 is true in $\ilt{1}{K_1}=K_2$. Assume Claim~2 holds for $\ilt{k-1}{K_1}$ for some $k>1$. Let $M_{k-1}$ be the desired matching in $\ilt{k-1}{K_1}$. Let $M_k=\{xy', x'y : xy\in M_{k-1}\}$. We then have that $M_k$ is a desired matching in $\ilt{k}{K_1}$, proving Claim~2.
	
Let $T$ be a spanning tree of $G$, and let $t_0=\lceil\log_2(\Delta(T))\rceil+1$. We will show that $G_{t_0}$ is Hamiltonian. Let $V(T)=\{v_1,v_2,\dots,v_{n_0}\}$, and let
\[
V_{i,t}=\{v\in V(G_t): v\text{ is a descendant of }v_i\}.
\]
Note that for all $t\geq 0$, $G_t[V_{i,t}]=\ilt{t}{K_1}$ and if $v_iv_j\in E(T)$, then $G_t[V_{i,t}\cup V_{j,t}]=\ilt{t+1}{K_1}$. By Claim~1, for each $1\leq i\leq n_0$, we can find a cycle $C_{(i)}$ that covers $V_{i,t}$, and has the property that the vertices in $V_{i,t-1}$ are adjacent to their clones in the cycle.
	
By Claim~2, for each edge $v_1v_j\in E(T)$, we can find a paired matching from $V_{i,t_0}$ to $V_{j,t_0}$. For each edge $v_1v_j\in E(T)$, given any vertex $v\in V_{i-1,t_0}$, there exists some $u\in V_{j-1,t_0}$ such that there is an edge switch between $vv'$ and $uu'$. This property guarantees that we can build a Hamiltonian cycle by iteratively using edge switches between edges of the form $uu'$ and $vv'$, one for each edge in $T$. As long as we never perform two edge switches at the same vertex, this will give us a Hamiltonian cycle. Since $t_0\geq\log_2(\Delta(T))+1$, we have that $|V_{i,t_0-1}|=2^{t_0-1}\geq \Delta(T)$. Thus, we can avoid performing two edge switches at the same vertex, so $G_{t_0}$ is Hamiltonian. Since $n_0-1\geq \Delta(T)$, the bound follows.
\end{proof}

As our final contribution, we consider the graphs which appear as induced subgraphs of ILM graphs. Interestingly, every finite graph eventually appears as an induced subgraph in $\ilm{t}SG$, regardless of the initial graph or sequence.

\begin{theorem}
If $F$ is a graph, then there exists some constant $t_0=t_0(F)$ such that for all $t\geq t_0$, all graphs $G$, and all binary sequences $S$, $F$ is an induced subgraph of $\ilm{t}SG$.
\end{theorem}

\begin{proof}
The proof will come in two stages. First, we will show that for some constant $k=k(F)$, $\ilt{k}{K_1}$ contains an induced copy of $F$, then we will show that for large enough constant $t_0=t_0(k)$, $\ilm{t_0}S{G}$ contains an induced copy of $\ilt{k}{K_1}$.
	
Let $|V(F)|=\ell$. First, we claim that $\ilt{\ell+\binom{\ell}2-1}{K_1}$ contains an induced copy of $F$. Towards induction, first note that the clique number $\omega(\ilt{\ell-1}{K_1})\geq\ell$ since at each time-step the clone of the original vertex is adjacent to all the other vertices that were clones of the original vertex, so $K_\ell$ is a subgraph of $\ilt{\ell-1}{K_1}$. We will show that if $F$ is an induced subgraph of $\ilt{r}{K_1}$, then $F-e$ is an induced subgraph of $\ilt{r+1}{K_1}$ for any edge $e\in E(F)$. Indeed, let $e=uv$ be an edge of $F$, and let $V(F)$ be the vertices in $\ilt{r}{K_1}$ that induce a copy of $F$. We then have that $(V(F)\setminus\{u,v\})\cup\{u',v'\}$ induces a copy of $F-e$ since $N_r[u]=N_{r+1}(u')$, and $N_r[v]=N_{r+1}(v')$, while $u'$ and $v'$ are non-adjacent. Via induction on the number of missing edges, and the fact that $\ilt{r}{K_1}$ is an induced subgraph of $\ilt{r+1}{K_1}$, we derive that $\ilt{\ell+\binom{\ell}2-1}{K_1}$ contains an induced copy of every graph on $\ell$ vertices.
	
We will show that for any fixed $k$, for any graph $G$, and any binary sequence $S$ $\ilm{2k}SG$ contains an induced copy of $\ilt{k}{K_1}$. First note that if $\ilm{t}SG$ contains an induced copy of $\ilt{r}{K_1}$ for any integer $r$, and we perform a transitive step, then $\ilm{t+1}SG$ contains an induced copy of $\ilt{r+1}{K_1}$ on the vertices from the copy of $\ilt{r}{K_1}$ along with their clones. We claim that if $\ilm{t}SG$ contains an induced copy of $\ilt{r}{K_1}$, and we perform two anti-transitive steps, then $\ilm{t+2}SG$ contains an induced copy of $\ilt{r+1}{K_1}$.
	
For each vertex $x\in V(\ilm{t}SG)$, let $x^{**}\in V(\ilm{t+2}SG)\setminus V(\ilm{t+1}SG)$ denote the anti-transitive clone of the anti-transitive clone of $x$. Note that the neighborhood $N_t[x]=N_{t+2}(x^{**})\cap V(\ilm{t}SG)$, so $V(\ilt{r}{K_1})$ along with the set of anti-transitive clones $\{x^{**}\in~V(\ilm{t+2}SG) : x\in~V(\ilt{r}{K_1})\}$ induce a copy of $\ilt{r+1}{K_1}$. Thus, any combination of $k$ transitive steps or back-to-back (pairwise-disjoint) pairs of anti-transitive steps will give us an induced copy of $\ilt{k}{K_1}$. Observe that such a combination must exist within the first $2k$ time-steps since if we partition the first $2k$ elements of the sequence $S$ into $k$ contiguous pairs, each pair that does not contain a $1$ contains two 0's, so each pair corresponds to at least one transitive step, or back-to-back anti-transitive steps. Thus, $\ilm{2k}SG$ contains an induced copy of $\ilt{k}{K_1}$. We then have that $\ilm{2(\ell+\binom{\ell}2-1)}SG$ contains an induced copy of $\ilt{\ell+\binom{\ell}2-1}{K_1}$, which in turn contains an induced copy of $F$, completing the proof.
\end{proof}

\section{Conclusion and future directions}

We introduced the Iterated Local Model (ILM) for social networks, which generalized the ILT and ILAT models previously studied in \cite{ilt1,ilt} and \cite{ilat}, respectively. We proved that graphs generated by ILM densify over time for any given sequence $S$. For the sequences $S$ with bounded gaps between zeros, the generated graphs are dense. The clustering coefficient of ILM graphs generated by sequences $S$ with bounded gaps between zeros, in contrast to ILT graphs, is bounded away from zero. We showed that the chromatic number of an ILM graph at time-step $t$ is the chromatic number of the initial graph plus either $t$ or $t-1$. The domination number of an ILM graph with at least two bits equal to zero eventually will be at most 3, and the diameter of an ILM graph becomes 3 eventually. The graphs generated by ILM exhibit bad spectral expansion as found in social networks. Structural properties of ILM graphs were also studied. We proved that after relatively few anti-transitive steps, ILM graphs become Hamiltonian. In addition, we showed that regardless of the sequence $S$ and the initial graph $G$ every graph eventually appears as an induced subgraph in $\ilm{t}SG$.

Several open problems remain concerning ILM graphs. While the eigenvalues of the ILT model are well-understood and have a recursive structure, much less is known about the eigenvalues of the ILAT model, and hence, for ILM graphs. We do not have a precise asymptotic value for the clustering coefficient of ILM graphs.

Another open direction concerns the domination number. In Theorem~\ref{theorem dom3}, we need two bits in $S$ to be $0$. It remains to determine what can occur after a single anti-transitive step. Any set $S\subseteq V(G)$ such that $S$ dominates $G$ and \emph{totally} dominates $\overline{G}$ (in total domination, vertices do not dominate themselves), will also dominate in $\mathrm{LAT}(G)$. It is unclear if minimizing this ``graph and complement'' domination parameter will give the correct domination number after one anti-transitive step. Note that if $G$ has a dominating vertex, then there is no total dominating set in $\overline{G}$.

Another direction we did not explore are graph limits for dense sequences arising from ILM graphs, such as for ILAT graphs. See \cite{L} for background on graph limits. The convergence of a graph sequence $G_n$ of simple graphs whose number of vertices tends to infinity is based on the homomorphism density of the graphs $G_n$.  Let $F$ be a fixed simple graph, then the homomorphism density of a graph $G$ is
\[
t(F,G) =\frac{\mathrm{hom}(F,G)}{|V (G)|^{|V (F)|}}
\]
where $\mathrm{hom}(F,G)$ denote the number homomorphisms of $F$ into $G$. The sequence $\{G_n\}$ converges if for every simple graph $F$, the sequence $\{t(F,G_n)\}$ converges. Since ILAT graphs are all eventually dense, we may study the convergence of the sequence of graphs generated by the ILAT model.

A randomized version of ILM would be an intriguing model. Duplication models \cite{CL}, where we randomize which edges to clone are difficult to rigorously analyze given their dependency structure. One approach would be to consider random binary input sequences $S$, and study properties of ILM graphs for such random sequences. Random sequences generated by flipping a fair coin will not have bounded gaps as we expect runs of $1$'s of length on the order of $\log(t)$ to occur, and so many of the results developed in this work do not immediately apply to these random models.


\begin{thebibliography}{99}

\bibitem{SPA} W.\ Aiello, A.\ Bonato, C.\ Cooper, J.\ Janssen, P.\ Pra\l{}at, A spatial web graph model with local influence regions, \emph{Internet Mathematics} (2009) \textbf{5} 175--196.

\bibitem{ba} A.\ Barab\'{a}si, R.\ Albert, Emergence of scaling in random networks, \emph{Science} \textbf{286} (1999) 509--512.

\bibitem{market} V.\ Boginski, S.\ Butenko, P.M.\ Pardalos, On structural properties of the market graph, In: A. Nagurney, editor, \emph{Innovation in Financial and Economic Networks}, Edward Elgar
    Publishers, pp.\ 29--45.

\bibitem{bol} B.\ Bollob\'{a}s, O.\ Riordan, J.\ Spencer, G.\ Tusn\'{a}dy, The degree sequence of a scale-free random graph process, \emph{Random Structures and Algorithms} \textbf{18} (2001) 279--290.

\bibitem{bbook} A.\ Bonato, \emph{A Course on the Web Graph}, American Mathematical Society Graduate Studies Series in Mathematics, Providence, Rhode Island, 2008.

\bibitem{surv} A.\ Bonato, N.\ Eikmeier, D.F.\ Gleich, R.\ Malik, Dynamic Competition Networks: detecting alliances and leaders, In: \emph{Proceedings of WAW'18}, 2018.

\bibitem{MGEOP} A.\ Bonato, D.F.\ Gleich, M.\ Kim, D.\ Mitsche, P.\ Pra\l{}at, A.\ Tian, S.J.\ Young. Dimensionality matching of social networks using motifs and eigenvalues, \emph{PLOS ONE}
    \textbf{9}(9):e106052, 2014.

\bibitem{ilt1} A.\ Bonato, N.\ Hadi, P.\ Pra\l{}at, C.\ Wang, Dynamic models of on-line social networks, In: \emph{Proceedings of WAW'09}, 2009.

\bibitem{ilt} A.\ Bonato, N.\ Hadi, P.\ Horn, P.\ Pra\l{}at, C.\ Wang, Models of on-line social networks, \emph{Internet Mathematics} \textbf{6} (2011) 285--313.

\bibitem{ilat} A.\ Bonato, E.\ Infeld, H.\ Pokhrel, P.\ Pra\l{}at, Common adversaries form alliances: modelling complex networks via anti-transitivity, In: \emph{Proceedings of WAW'17}, 2017.

\bibitem{geop} A.\ Bonato, J.\ Janssen, P.\ Pra\l{}at, Geometric protean graphs, \emph{Internet Mathematics} \textbf{8} (2012) 2--28.

\bibitem{BJR} A.~Bonato, J.~Janssen, E.~Roshanbin, How to burn a graph, \emph{Internet Mathematics} \textbf{1-2} (2016) 85--100.

\bibitem{at} A.\ Bonato, A.\ Tian, Complex networks and social networks, invited book chapter in: \emph{Social Networks},  editor E. Kranakis, Springer, Mathematics in Industry series, 2011.
\bibitem{broder} A.\ Broder, R.\ Kumar, F.\ Maghoul, P.\ Raghavan, S.\ Rajagopalan, R.\ Stata, A.\ Tomkins, J.\ Wiener, Graph structure in the web, \emph{Computer Networks} \textbf{33} (2000) 309-320.

\bibitem{sgt} F.R.K.\ Chung, {\it Spectral Graph Theory}, American Mathematical Society, Providence, Rhode Island, 1997.

\bibitem{CL} F.R.K.\ Chung, L.\ Lu, \emph{Complex Graphs and Networks}, American Mathematical Society, U.S.A., 2004.

\bibitem{ek} D.\ Easley, J.\ Kleinberg, \emph{Networks, Crowds, and Markets Reasoning about a Highly Connected World},
Cambridge University Press, 2010.

\bibitem{estrada} E.\ Estrada, Spectral scaling and good expansion properties in complex networks, \emph{Europhys. Lett.} \textbf{73} (2006) 649--655.

\bibitem{guo} W.\ Guo, X.\ Lu, G.M.\ Donate, S.\ Johnson, The spatial ecology of war and peace, Preprint 2019.

\bibitem{k} D.\ Krioukov, F.\ Papadopoulos, M.\ Kitsak, A.\ Vahdat, M.\ Bogu\~n\'a, Hyperbolic geometry of complex networks, \emph{Phys. Rev. E}, \textbf{82}:036106, 2010.

\bibitem{les1} J.\ Leskovec, J.\ Kleinberg, C.\ Faloutsos, Graphs over time:
densification Laws, shrinking diameters and possible explanations,
In: \emph{ Proceedings of the 13th ACM SIGKDD International
Conference on Knowledge Discovery and Data Mining}, 2005.

\bibitem{L} L.\ Lov\'asz, \emph{Large networks and graph limits}, American Mathematical Society,  Providence, RI, 2012.

\bibitem{mmp} V.\ Memi\v{s}evi\'{c}, T.\ Milenkovi\'{c}, N.\ Pr\v{z}ulj, An integrative approach to modeling biological networks, \emph{Journal of Integrative Bioinformatics},
    \textbf{7}:120, 2010.

\bibitem{scott} J.P.\ Scott, \emph{Social Network Analysis: A Handbook}, Sage Publications Ltd, London, 2000.

\bibitem{mason} L.\ Small, O.\ Mason, Information diffusion on the iterated local transitivity model of online social networks, \emph{Discrete Applied Mathematics} \textbf{161} (2013) 1338–1344.

\bibitem{west} D.B.\ West, \emph{Introduction to Graph Theory, 2nd edition}, Prentice Hall, 2001.

\bibitem{zachary} W.W.\ Zachary, An information flow model for conflict and fission in small groups, \emph{Journal of Anthropological Research} \textbf{33} (1977) 452--473.

\bibitem{Zhang} Z.\ Zhang, F.\ Comellas, G.\ Fertin, L.\ Rong, Highdimensional apollonian networks, \emph{Journal of Physics A: Mathematical and General}, \textbf{39}:1811, 2006.

\end{thebibliography}
\end{document}